\documentclass[10pt]{article}

\usepackage{times}
\usepackage{array, amssymb, amsmath, graphics}
\newtheorem{remark}{Remark}
\newtheorem{theorem}{Theorem}
\newtheorem{example}{Example}
\newtheorem{conj}{Conjecture}
\newtheorem{ques}[conj]{Question}
\newtheorem{lemma}{Lemma}
\newtheorem{proposition}{Proposition}

\newenvironment{proof}{{\bf Proof.}}{\hspace*{1mm}\hfill\rule{2mm}{2mm}}

\newtheorem{pretheorema}{{\bf Theorem}}

\def\vol{\rm vol}

\def\M#1#2#3{\mbox{\rm M}(#1\mbox{-}(#2,#3))}
\def\T#1#2#3{#1\mbox{-}(#2,#3)\mbox{ \rm Latin trade}}
\def\SFT#1#2#3{#1\mbox{-}(#2,#3)\mbox{ \sf Latin trade}}
\def\ITT#1#2#3{#1\mbox{-}(#2,#3)\mbox{ \it Latin trade}}

\def\Ts#1#2#3{#1\mbox{-}(#2,#3)\mbox{ \rm Latin trades}}
\def\ITTs#1#2#3{#1\mbox{-}(#2,#3)\mbox{ \it Latin trades}}
\def\m#1#2{\raise 0.2ex\hbox{
    ${#1_{\displaystyle #2}}$}}
\def\xx#1{\raise 0.5ex\hbox{
    ${#1}$}}

\title{Possible volumes of $\Ts{t}{v}{t+1}$}
\author{E. S. Mahmoodian\footnote{Department of Mathematical Sciences, Sharif University
of Technology, and School of Mathematics, Institute for Studies in
Theoretical Physics and Mathematics (IPM),  P. O. Box: 19395-5746
Tehran, Iran ({\tt emahmood@sharif.edu}).} \ and M. S.
Najafian\footnote{
 and
 Zanjan University, Zanjan, Iran ({\tt najafsm@yahoo.com}).} $^*$}
\date{}
\begin{document}

\maketitle
\begin{abstract}
The concept of $t$-$(v,k)$~trades of block designs previously has
been studied in detail. See for example~A. S. Hedayat (1990) and
Billington (2003). Also Latin trades have been studied in detail
under various  names, see~A. D. Keedwell (2004) for a survey.
Recently Khanban, Mahdian and Mahmoodian have extended the concept
of Latin trades and introduced $\Ts{t}{v}{k}$. Here we study the
spectrum of possible volumes of these trades, $S(t,k)$. Firstly, similarly to
 trades of block designs we consider $(t+2)$ numbers
$s_i=2^{t+1}-2^{(t+1)-i}$, $0\leq i\leq t+1$, as critical points
and then we show that  $s_i\in S(t,k)$, for any $0\leq i\leq t+1$,
\ and if $s\in (s_i,s_{i+1}),0\leq i\leq t$, then $s\notin
S(t,t+1)$.  As an example, we
determine $S(3,4)$ precisely. 
\end{abstract}


\maketitle


\section{Introduction and Preliminaries}
\noindent Let $V := \{1, 2, \ldots, v\}$ and $V^k$ be the set of
all ordered $k$-tuples of the elements of $V$, i.e.\ $V^k :=
\{(x_1, \ldots, x_k) \mid x_i\in V, i=1, \ldots, k\}$. Also, let
$V_I^t :=\{(u_1,\ldots,u_t)_I\mid u_i\in V, i=1, \ldots, t\}$,
where $I$ is a $t$-subset of  $\{1,\ldots,k\}$. For a pair of
elements of $V^k$ and $V_I^t$, where $I= \{i_1,\ldots,i_t\}$ and
$i_1 < \cdots < i_t$, we define
\[
(u_1,\ldots,u_t)_I\in (x_1,\ldots, x_k) \quad \Longleftrightarrow
\quad u_j=x_{i_j},\qquad j=1,\ldots,t.
\]

Next we define {\sf $t$-inclusion matrix} $M = \M{t}{v}{k}$, as
in~\cite{KhanbanMahdianMahmoodian}. The columns of this matrix
correspond to the elements of $V^k$ (in lexicographic order) and its
rows correspond to the elements of $\cup_I V_I^t$, where the union
is over all $t$-subsets of $\{1,\ldots,k\}$. The entries of this
matrix are 0 or 1, and are defined as follows.
\[
{\rm M}_{(u_1,\ldots,u_t)_I, (x_1,\ldots,x_k)} = 1 \quad
\Longleftrightarrow \quad (u_1,\ldots,u_t)_I\in (x_1,\ldots, x_k).
\]

A \ $\SFT{t}{v}{k}$  $T=(T_1,T_2)$ of {\sf volume} $s$ consists of
two disjoint collections $T_1$ and $T_2$, each of $s$ elements from
$V^k$, such that for each  $t$-set $I \subseteq \{1,\ldots,k\}$, and
for every element $(u_1,\ldots ,u_t)_I$ of $V_I^t$, the number of
elements of $T_1$ and $T_2$ that contain $(u_1,\ldots ,u_t)_I$ is
the same. Note that in checking the containment of an element
$(u_1,\ldots ,u_t)_I$, elements of $I$ are arranged in increasing
order. The volume of a Latin trade $T$ is denoted by ${\rm vol}(T)$.
It is clear from the definition above, that  for any $t' \le t$,
every \ $\T{t}{v}{k}$ \ is also a  $\T{t'}{v}{k}$. For simplicity,
the notation of $t$-Latin trade is commonly used for this
combinatorial object. The {\sf spectrum} of $\T{t}{v}{k}$s, $S(t,k)$
is the set of all integers $s$, such that for each $s$ there exists
a $\T{t}{v}{k}$ of volume $s$. A $\T{t}{v}{k}$ of volume $0$ is
considered always to exist, that is a trade with $T_1=T_2=\emptyset$
which will be called {\sf trivial trade}. In a $\T{t}{v}{k}$
$T=(T_1,T_2)$ both collections $T_1$ and $T_2$ cover the same
elements. This set of elements is called the {\sf foundation} of
 $T$ and is denoted by ${\rm found}(T)$.  Note that $v$ can be any
integer such that $v$ is at least the size of the
foundation of $T$.
\begin{example}\label{3-(3,4)l.t}
In the following a $\ITT{3}{3}{4}$ $T=(T_1,T_2)$ of volume $15$ and
with   ${\rm found}(T)=\{1,2,3\}$  is given.
\end{example}

\begin{center}
\begin{tabular}{|c|c|} \hline
$T_1$ &
\begin{tabular}{ccccccccccccccc}
    3&3&2&2&2&1&1&2&2&1&1&3&3&2&2 \\
    3&2&3&2&1&2&1&2&1&2&1&3&2&3&2 \\
    3&3&3&3&3&3&3&2&2&2&2&1&1&1&1 \\
    2&3&3&1&2&2&1&2&1&1&2&3&2&2&3
\end{tabular}
\\ \hline
\end{tabular}  \\[.5cm]
\begin{tabular}{|c|c|} \hline
 $T_2$ &
\begin{tabular}{ccccccccccccccc}
    3&3&2&2&2&1&1&2&2&1&1&3&3&2&2 \\
    3&2&3&2&1&2&1&2&1&2&1&3&2&3&2 \\
    3&3&3&3&3&3&3&2&2&2&2&1&1&1&1 \\
    3&2&2&3&1&1&2&1&2&2&1&2&3&3&2
\end{tabular}
\\ \hline
\end{tabular}
\end{center}
%
As it is noted in~\cite{KhanbanMahdianMahmoodian}, the set of all
$\Ts{t}{v}{k}$ is a subset of  the null space of the $t$-inclusion
matrix $M = \M{t}{v}{k}$. Also $t$-Latin trades have a close
relation with orthogonal arrays. For example, the intersection
problem of two orthogonal arrays may be studied as a problem in
$t$-Latin trades.

 One of the important questions
is:
\begin{ques}
\label{spectrumquestion1}%
What is the spectrum of $\ITTs{t}{v}{k}$?
\end{ques}
Similar question about the spectrum of trades of block designs was
raised in~\cite{MR1196125}, and two basic conjectures were
stated. Since then many results on this subject are published. For
a survey see \cite{MR1056530} and~\cite{MR2041871}.

The special case of $\Ts{2}{v}{3}$ is previously studied in detail
and is referred with different names such as ``disjoint and mutually
balanced'' (DMB) partial Latin squares by Fu and Fu (see for
example~\cite{MR1125351}), as an ``exchangeable partial groupoids''
by Dr{\'a}pal and Kepka~\cite{MR733686}
 as a ``critical partial Latin square''
(CPLS)  by Keedwell~(\cite{MR96a:05027} and \cite{MR1393712}), and
as a ``Latin interchange'' by Diane Donovan
et~al.~\cite{MR98b:05019}, and recently as a ``Latin bitrade''
by~Dr{\'a}pal et~al.~(see\cite{Drapal}, \cite{MR2338087},
and~\cite{Hamalainen}). See~\cite{CavenaghMathSlovac} for a recent
survey.

Following~\cite{MR2338087} we will refer them
     as Latin
bitrades. Let $L_1$ and $L_2$ be two Latin squares of the same
order $n$. A {\sf Latin bitrade} $T=(P,Q)$ consists of two
partial Latin squares $P$ and $Q$ obtained from $L_1$ and $L_2$,
respectively, by deleting their common entries.
Note that $\Ts{2}{v}{3}$ are more general than Latin bitrades:
in the former, repeated blocks and multiple symbols
in rows, columns and cells are allowed.
\begin{example}\label{2-(3,3)l.t}
The following is a Latin bitrade of volume $7$. (It should be noted
that one empty row and one empty column are deleted.)
\end{example}
\begin{center}
\begin{tabular}{
 |@{\hspace{1pt}}c@{\hspace{1pt}}
 |@{\hspace{1pt}}c@{\hspace{1pt}}
 |@{\hspace{1pt}}c@{\hspace{1pt}}
 |}
\hline
\m12 & \m21 & \xx.  \\
\hline
\m21 & \m13 & \m32  \\
\hline
\xx. & \m32 & \m23  \\
\hline
\end{tabular}
\end{center}
A result in~\cite{fu} answers the  Question~\ref{spectrumquestion1}
in the special case of Latin bitrades. Here we state several
theorems about existence and nonexistence of $\Ts{t}{v}{k}$ of
specified volumes, and we determine the  spectrum of
$\Ts{t}{v}{t+1}$ for $t=1,2,$ and $3$.
%
%
\section{Possible volumes of $t$-Latin trades}
%
Most of the concepts and definitions about $\Ts{t}{v}{k}$ are
borrowed from $t$-$(v,k)$~trades of block designs. For example:
volume, spectrum, $t$-inclusion matrix, frequency vector, etc.
Specially, we show that there are close relations between the
spectrum of these two combinatorial objects. But, in spite of all
the similarities, some differences are observed between them, both
in properties and in the methods of proof of lemmas and theorems.

By the following lemma all existence results of
$\Ts{t}{v}{t+1}$ can be extended to $\Ts{t}{v}{k}$.
\begin{lemma}\label{t+1tok}
For any $k \ge t+1$,  we have $S(t,t+1)\subseteq S(t,k)$.
\end{lemma}
\begin{proof}{ Let $T=(T_1,T_2)$ be a $\T{t}{v}{t+1}$ of
volume $s$. For each ordered $(t+1)$-tuple  in $T_1$ and $T_2$, we
add $(k-t-1)$ fixed elements  $x$ of $V$ as $(t+2)^{\rm nd \/}$ to
$k^ {\rm th \/}$ coordinates. Then we obtain two collections $T_1^*$
and $T_2^*$ containing of ordered $k$-tuples. Clearly
$T^*=(T_1^*,T_2^*)$ is a $\T{t}{v}{k}$ of volume $s$.~}\end{proof}
%
\begin{lemma}\label{2s}
By using any $\ITT{t}{v}{k}$ of volume $s$, we can obtain a \linebreak
$\ITT{(t+1)}{v}{k+1}$ of volume $2s$.
\end{lemma}
\begin{proof}{  Let $T=(T_1,T_2)$ be a $\T{t}{v}{k}$ of
volume $s$. Choose two distinct elements $x$ and $ y \in V$. The
following construction (see Figure~1) produces a
$\T{(t+1)}{v}{k+1}$ $T^*=(T_1^*,T_2^*)$ of volume $2s$. That is, for
constructing $T^*$ we adjoin two new distinct symbols $x$ and $y$
(respectively) to the first component of each element of
    $T_1$    and   $T_2$   (respectively), to obtain     $T_1^*$    and   $T_2^*$
   (respectively).
\begin{center}
\hspace*{-4.8cm}
$T^*_1$  \\[-.55cm]
\hspace*{5.1cm}
$T^*_2$  \\ [.3cm]
%
\begin{tabular}{|c|ccccccc|}   \hline

 x       &    &    &    &        &    &    &           \\
 x       &    &    &    &        &    &    &           \\
 .       &    &    &    &        &    &    &           \\
 .       &    &    &    & $T_1$  &    &    &           \\
 .       &    &    &    &        &    &    &           \\
 x       &    &    &    &        &    &    &           \\    \hline
 y       &    &    &    &        &    &    &           \\
 y       &    &    &    &        &    &    &           \\
 .       &    &    &    &        &    &    &           \\
 .       &    &    &    & $T_2$  &    &    &           \\
 .       &    &    &    &        &    &    &           \\
 y       &    &    &    &        &    &    &           \\    \hline
 \end{tabular}        \qquad
 \begin{tabular}{|c|ccccccc|}   \hline
   x       &    &   &    &            &    &     &           \\
   x       &    &   &    &            &    &     &           \\
   .       &    &   &    &            &    &     &           \\
   .       &    &   &    & $T_2$      &    &     &           \\
   .       &    &   &    &            &    &     &           \\
   x       &    &   &    &            &    &     &           \\    \hline
   y       &    &   &    &            &    &     &           \\
   y       &    &   &    &            &    &     &           \\
   .       &    &   &    &            &    &     &           \\
   .       &    &   &    & $T_1$      &    &     &           \\
   .       &    &   &    &            &    &     &           \\
   y       &    &   &    &            &    &     &           \\    \hline
\end{tabular}    \\[1cm]
{\bf Figure \ 1}  \\[1.5cm]
\end{center}
\vspace*{-2cm}
}\end{proof}
\begin{remark}\label{union}

Assume we have two $\ITTs{t}{v}{k}$, $T=(T_1,T_2)$  and \
$R=(R_1,R_2)$.  Then \ $T+R=(T_1\cup R_1,T_2\cup R_2)$ \
 and \ $T-R=(T_1\cup R_2,T_2\cup R_1)$ \ are two $\ITTs{t}{v}{k}$.
 Note that the elements which appear in both sides are omitted. So
 $T+R$ and  $T-R$ are of
 volumes
 $|T_1|+|R_1|-|T_1\cap R_2|-|T_2\cap R_1|$ and
 $|T_1|+|R_2|-|T_1\cap R_1|-|T_2\cap R_2|$, respectively.
\end{remark}
\begin{remark}\label{corresponding}
If we look at each ordered $k$-tuple in  $T_i$ and $R_i$ , $i=1$,
$2$, as a variable (each element of $T_1$ and $R_1$ with positive sign
and each element of $T_2$ and $R_2$ with negative sign), then the
two operations above coincide with the concept of two algebraic
 $+$ and $-$ operations. For this reason sometimes we denote a
 $\ITT{t}{v}{k}$
$T=(T_1,T_2)$ as $T=(T_1-T_2)$.
\end{remark}
To apply linear algebra, we correspond to each $\T{t}{v}{k}$
$T=(T_1,T_2)$, a frequency vector {\bf T}, where the components
of {\bf T} are corresponded with all elements of $V^k$ (in
lexicographic order). For each $x\in V^k$, {\bf T}$(x)$ is
defined as in the following: \\ [.4cm] ${\bf T}(x)=\left
\{\begin{array}{ll}
\hspace*{3.4mm} p    &       \hspace{3cm} {\rm if} \ x\in T_1  \   (p \ {\rm times}),  \\
-q    &     \hspace{3cm} {\rm if} \ x\in T_2  \  (q  \ {\rm times}),  \\
\hspace*{3.4mm} 0    &     \hspace{3cm} {\rm    otherwise.}   \end{array}
\right. $  \\[.1cm]

Let {\rm M} be the $t$-inclusion matrix $M = \M{t}{v}{k}$. Then
it is an easy exercise to prove that  {\rm M}{\bf
T}=$\overline{\bf 0}$, where $\overline{\bf 0}$ is the zero
vector. And conversely if {\bf T}, with integer components,
is a vector in the  null space
of {\rm M}  then it determines a
$\T{t}{v}{k}$ $T=(T_1,T_2)$. $T_1$ is obtained from the positive
components and $T_2$ is obtained from the negative components of
vector {\bf T}. In other words, there is a one-to-one
correspondence between the null space of {\rm M} over the ring
$\mathbb Z$ and the set of all $\Ts{t}{v}{k}$. The following lemma
is the fact mentioned in Remark~\ref{union}, but in a linear
algebraic approach.
%
\begin{lemma}\label{T+R}
Consider two $\ITTs{t}{v}{k}$, $T=(T_1-T_2)$  and  $R=(R_1-R_2)$.
Then $T+R=(T_1+R_1)-(T_2+R_2)$ is also a $\ITT{t}{v}{k}$.
\end{lemma}
\begin{proof}{ Let {\bf T} and {\bf R} be the frequency vectors of
$T$ and $R$, respectively, and {\rm M} be the
$t$-inclusion matrix $M = \M{t}{v}{k}$. We have {\rm M}{\bf T}=
$\overline{\bf 0}$ and {\rm M}{\bf R}= $\overline{\bf 0}$. Thus
{\rm M}{\bf (T+R)}= $\overline{\bf 0}$, i.e. {\bf (T+R)} belongs
to the null space of {\bf M}. Therefore $T+R$ is a $\T{t}{v}{k}$.
}\end{proof}
\begin{remark}
In the previous lemma if \ $T_1\cap R_2=R_1\cap T_2=\emptyset$, then
$\vol (T+R)= \vol  (T)+ \vol  (R)$.
\end{remark}
In~\cite{KhanbanMahdianMahmoodian},
a
$\T{t}{v}{k}$  is represented by  a homogeneous polynomial of order $k$ as follows.
Let $P=P(x_1,x_2,\ldots ,x_{v})$ be a homogeneous polynomial of
order $k$ whose terms are ordered multiplicatively (meaning that for
example for $i_1 \neq i_2$ the term $x_{i_1}x_{i_2}x_{i_3}\cdots
x_{i_k}$ is different from $x_{i_2}x_{i_1}x_{i_3}\cdots x_{i_k}$,
etc.) Now we correspond a frequency vector {\bf T}, with $v^k$
components (in lexicographic order) to polynomial $P$ as in the
following:

For $x=(i_1,i_2,\ldots ,i_k)\in V^k$ we let {\bf T}($x$) be the
coefficient of $x_{i_1}x_{i_2}x_{i_3}\cdots x_{i_k}$ in $P$. So, if
the resulting vector {\bf T} satisfies the equation {\rm M}{\bf T}=
$\overline{\bf 0}$,  then we refer to polynomial $P$ as a
$\T{t}{v}{k}$. It is easy to show that this definition is equivalent
to the previous definition of $\T{t}{v}{k}$. This representation
 helps us in constructing
$\Ts{t}{v}{k}$ of desired volumes.

The following theorem is
proved by using polynomial representation  of \linebreak $\Ts{t}{v}{k}$.
%

%
\begin{theorem}\label{existencesi}
For each \  $s_i=2^{t+1}-2^{(t+1)-i}$, \ $0\leq i\leq t+1$, \
there exists a $\ITT{t}{v}{k}$ of volume $s_i$ with $k\geq t+1$.
\end{theorem}
\begin{proof}{
For \ $i=0$ \ the trivial trade is the answer. For each \ $i$,
$1\leq i\leq t+1$, let $T=(T_1-T_2)$ and $R=(R_1-R_2)$ be two
$\Ts{t}{v}{k}$ defined as follows:

 $T=T_1-T_2$ \\
\hspace*{.36cm}$=(x_1-x_2)\cdots(x_{2t-2i+1}-x_{2t-2i+2})(x_{2t-2i+3}-x_{2t-2i+4})\cdots$

\hspace*{.73cm}$(x_{2t+1}-x_{2t+2})x_{2t+3}\cdots x_{k+t+1}$,
\quad and
\vspace*{6mm} \\
$R=R_1-R_2$ \\ \hspace*{.36cm}$=-(x_1-x_2)\cdots
(x_{2t-2i+1}-x_{2t-2i+2})(y_{2t-2i+3}-x_{2t-2i+4})\cdots$  \\
\hspace*{.93cm} $(y_{2t+1}-x_{2t+2})x_{2t+3}\cdots x_{k+t+1}$, \\

where inside each parenthesis variables are different from each
other, and also for each $j$, \ $y_{j} \neq x_{j}$. Now \ $T+R$ \
is a $\T{t}{v}{k}$, by Lemma~\ref{T+R}. $T$~and $R$ are the same
in $((t+1)-i)$ parentheses. So, in $T+R$, the following terms are
cancelled out with their negatives:
\\
$(x_1-x_2)\cdots(x_{2t-2i+1}-x_{2t-2i+2})x_{2(t-i+2)}\cdots
x_{2(t+1)}x_{2t+3}\cdots x_{k+t+1}$. 
Thus $T+R$ is a $\T{t}{v}{k}$ of volume
$s_i=2^{t+1}-2^{(t+1)-i}$. }
\end{proof}


To continue our discussion we need to define  levels of a trade.  We may
decompose a $t$-Latin trade $T$ and obtain other $(t-1)$-Latin
trades. Let $T=(T_1,T_2)$ be a $\T{t}{v}{k}$ and let $j \in
\{1,\ldots,k\}$ and $x \in V$.  Take \ $T'_i=\{(x_1, \ldots ,x_{k}) |
(x_1, \ldots ,x_{k})\in T_i \ {\rm and} \ x_j=x\}$, for $i=1$,
$2$. Delete $x$ from the $j^{\rm th \/}$ coordinate in all elements of $T'_1$ and $T'_2$
to obtain $T''_1$ and $T''_2$, respectively. Now $T''=(T''_1,T''_2)$  is a
 $\T{(t-1)}{v'}{k-1}$,
which is called a {\sf level trade} of $T$ in the direction
of~$j$.

\begin{example}\label{3level trade}
In  Example \ref{3-(3,4)l.t} for $j=3$, there exist three level
trades. For example, for $x=3$ the level trade in the direction
of~$j=3$ is as follows.
\end{example}
\begin{center}
\begin{tabular}{|c|ccccccc|}        \hline

   \     &3&3&2&2&2&1&1      \\
$T''_1$  &3&2&3&2&1&2&1      \\
   \     &2&3&3&1&2&2&1      \\     \hline

\end{tabular}  \\[.5cm]
%
\begin{tabular}{|c|ccccccc|}       \hline
         &3&3&2&2&2&1&1      \\
$T''_2$  &3&2&3&2&1&2&1      \\
         &3&2&2&3&1&1&2      \\    \hline

\end{tabular}  
\end{center}
Note that the level trade above is a Latin bitrade, which  also can
be represented as in Example~\ref{2-(3,3)l.t}.
\begin{lemma}\label{level}
Let $T=(T_1,T_2)$ be a $\ITT{t}{v}{t+1}$ of volume $s$ with only
two non-trivial level trades in some direction $j$. Then the
volume of these  level trades are equal, say to $a$, and so
$s=2a$.
\end{lemma}
\begin{proof}{ Without loss of generality assume $j=1$. It is easy to see that
the structure of $T=(T_1,T_2)$ is the same as structure of $T^*$ in Figure~1, where
$k=t+1$. So the two level trades of $T$ in the direction of $j=1$
have the same volume $a$.  Moreover, if $T'=(T_1,T_2)$ is one of
these level trades, then the other level trade is $T''=(T_2,T_1)$.
}\end{proof}

Now we investigate the spectrum of $\Ts{t}{v}{t+1}$.
\begin{proposition}\label{S(1,2)} $S(1,2)=\mathbb N_0\backslash \{1\}$.
\end{proposition}
\begin{proof}{ It is clear that a $\T{1}{v}{2}$ of volume
$1$ does not exist. Suppose $s\geq 2$, the following array form a
$\T{1}{v}{2}$ of volume $s$.
\begin{center}
\begin{tabular}{|c|c|} \hline
$T_1$ &
\begin{tabular}{cccccc}
        1&2&3&$\cdots$ &$s-1$&$s$     \\
        1&2&3&$\cdots$ &$s-1$&$s$
\end{tabular}
\\ \hline
\end{tabular}  \\[.5cm]
\begin{tabular}{|c|c|} \hline
 $T_2$ &
\begin{tabular}{cccccc}
        1&2&3&$\cdots$ &$s-1$&$s$    \\
        2&3&4&$\cdots$ &$s$&1
\end{tabular} 
\\ \hline
\end{tabular}
\end{center}
\vspace*{-6.5mm} }\end{proof}

The following result of H-L. Fu. is an instrument in building an
induction base.
\begin{proposition}\label{Fu}~{\rm \cite{fu}}
A Latin bitrade  $T=(P,Q)$ of volume $s$ exists if and only if $s\in
\mathbb N_0\backslash \{1,2,3,5\}$.
\end{proposition}

\begin{proposition}\label{S(2,3)}
$S(2,3)=\mathbb N_0\backslash \{1,2,3,5\}.$
\end{proposition}
\begin{proof}{ Obviously, there exist no $\Ts{2}{v}{3}$ of
volumes $1$ and $2$. Assume that $T$ is a $\T{2}{v}{3}$ of volume
$3$ (or $5$). Then by Lemma~\ref{level}, each of these two numbers must decompose into at
least three positive numbers from the set $S(1,2)=\{0,2,3,4,\ldots
\}$ which is impossible.  \\ Each  Latin bitrade is a $\T{2}{v}{3}$,
so $\mathbb N_0\backslash \{1,2,3,5\}\subseteq S(2,3).$ }
\end{proof}
\begin{theorem}\label{nonexistence<s0}
There  exists no $\ITT{t}{v}{t+1}$ of volume $s$, for any \linebreak
$s_0=0<s<2^t=s_1$.
\end{theorem}
\begin{proof}{  We proceed by induction on $t$. The
statement obviously holds for the case $t=1$. Assume, by induction
hypotheses, the statement holds for all values less than $t$, i.e.
if $a$ is the volume of a $t'$-Latin trade ($t'<t$), then $a \geq
2^{t'}$. We show that theorem holds for $t$ also. Suppose the
statement is not true for $t$, and let $T$ be a  $\T{t}{v}{t+1}$ of
volume $s$ with $0<s<2^t$. $T$ has at least two non-trivial level
trades in each direction. Suppose in some direction $j$, \ $T$ has $l$
level trades of volumes $ a_1, a_2, \ldots, a_l$, where $l \geq 2$
and $s=a_1+\cdots +a_l$. By induction hypotheses $a_i \geq 2^{t-1}$,
for each $i$. Therefore $s \geq l\cdot 2^{t-1} \geq 2\cdot  2^{t-1}
=2^{t}$, which is a contradiction. }\end{proof}

\begin{theorem}\label{nonexistencesi}
For any  \ $s\in (2^{t+1}-2^{(t+1)-i} ,2^{t+1}-2^{(t+1)-(i+1)}),
\linebreak 1\leq i\leq t$, there does not exist any \
$\ITT{t}{v}{t+1}$ of volume~$s.$
\end{theorem}
\begin{proof}{ We proceed by induction on $t$.  For
case $t=1$, there is nothing to be proved. For $t=2$, statement
follows from Proposition~\ref{S(2,3)}. Assume, by induction
hypotheses, that statement holds for all values less than~$t$
($t>2$), i.e. if $t'<t$ then there exists no $\T{t'}{v'}{t'+1}$ of
volume $s'$, where
$s'_i=2^{t'+1}-2^{(t'+1)-i}<s'<2^{t'+1}-2^{(t'+1)-(i+1)}=s'_{i+1},
1\leq i\leq t'$.
We show that it holds for $t$ also. Suppose in contrary for some
$i$ and some $s$, where $ s_i<s<s_{i+1}$, there exists a
$\T{t}{v}{t+1}$ of volume $s$. We show a contradiction.
There are
three cases to consider:

{\bf Case $1.$} In some direction $T$ has only two non-trivial
level trades. So by Lemma~\ref{level} we have  $s=2s'$, where
$s'$ is the volume of some $\T{(t-1)}{v'}{t}$. Therefore  we have
$\frac {s_i}{2}<s'=\frac {s}{2} < \frac {s_{i+1}}{2}$, or %
$$2^{(t-1)+1}-2^{[(t-1)+1]-i}
<s'<2^{(t-1)+1}-2^{[(t-1)+1]-(i+1)}$$ which is a contradiction.

{\bf Case $2.$} In each direction $T$ has more than two
non-trivial level trades, and in some direction it has only three
non-trivial level trades. So $s=a+b+c$, where for each value  of
$a$,  $b$ and $c$ there exist $\Ts{(t-1)}{v'}{t}$ of these
volumes. Note that by Theorem~\ref{nonexistence<s0} we have $a,
b, c\geq 2^{t-1}.$ We claim that at
least two of values $a$, $b$ and $c$ are equal to $2^{t-1}$. \\
Proof of claim:  We know that the critical points in the case
$t-1$,  in increasing order, are
$$s'_o=0,s'_1=2^{t-1},s'_2=3\cdot 2^{t-2},s'_3=7\cdot 2^{t-3},\ldots,
s'_{t}=2^t-1.$$ If $a=2^{t-1}$ and $b$, $c\geq 3\cdot 2^{t-2}$,
then
$$s=a+b+c\geq 2^{t-1}+2\cdot 3\cdot 2^{t-2}=2^{t-1}(1+3)=2^{t+1},$$
which is impossible, because, $s<s_{t+1}=2^{t+1}-1$. So we have
either
\begin{itemize}
\item[a)]
$a=b=2^{t-1}$ and $c=2^t-2^{t-j}$, for some $j$, \ $1\leq j\leq t$ \ \
or
\item[b)]
$a=b=2^{t-1}$ and $c>2^{t}-1$.
\end{itemize}
In $(a)$ we have  $s=a+b+c=2\cdot
2^{t-1}+2^t-2^{t-j}=2^{t+1}-2^{(t+1)-(j+1)}.$ \linebreak
 This means that $s$
is a critical point of case $t$, which is a contradiction. In
$(b)$ we have \ $s=a+b+c>2\cdot 2^{t-1}+2^{t}-1=2^{t+1}-1$, which
is also impossible.

{\bf Case $3.$} In all directions $T$  has at least four
non-trivial level trades. This means that $s=\sum_{i=1}^{l}a_i$,
where $l\geq4$ and for each $a_i$ there exists a
$\T{(t-1)}{v'}{t}$ of volume $a_i$. But then we have $s\geq
4\cdot 2^{t-1}=2^{t+1}$, which is impossible. }\end{proof}

\section{Spectrum of  $\Ts{3}{v}{4}$}
For two integers $a$ and $b$ with $a<b$ we denote
$[a,b]=\{a,a+1,\ldots ,b\}$.  We prove the following theorem.
%
\begin{theorem}\label{S(3,4)}
$S(3,4)=\mathbb N_0\backslash ([1,7]\cup [9,11]\cup \{13\}).$
\end{theorem}
\begin{proof}{ By Lemma~\ref{2s} and Proposition~\ref{S(2,3)}, for each even
number \linebreak $s \in \mathbb N_0\backslash ([1,7]\cup [9,11]\cup
\{13\})$ we can construct a $\T{3}{v}{4}$ of volume $s$. A
$\T{3}{v}{4}$ of volume 15 is given in Example~\ref{3-(3,4)l.t}
and  $\Ts{3}{v}{4}$ of volumes 17, 19, and 21 are given in the
Appendix. $\Ts{3}{v}{4}$ of volumes 23 and 25 may be constructed
by combination of $\Ts{3}{v}{4}$ of volumes (8 and 15) and (8 and
17), respectively (Lemma~\ref{T+R}). So, up to this point we know
that for each $s(k)=2k+1$, where $ 7 \leq  k  \leq  12$, there
exists a $\T{3}{v}{4}$ of volume $s(k)$. For $k \geq 13$ we write
$s(k)=s(k-4)+8$, and then, by induction and by Lemma~\ref{T+R},
for each $k \geq 13$ we can obtain a $\T{3}{v}{4}$ of volume
$s(k)$. Now the proof  is complete by
Theorems~\ref{nonexistence<s0} and~\ref{nonexistencesi}.
}\end{proof}
\section{Future Research}
The study of $\Ts{t}{v}{k}$ when $k=t+1$, is of special interest.
For example similar to Latin  bitrades, some $\Ts{3}{v}{4}$  may
also be denoted by $T=(M,N)$, where $M$ and $N$ are two partial
Latin cubes obtained from some Latin cubes $C_1$ and $C_2$ by
deleting their common entries. This geometrical view will shed a
light to studying questions and conjectures about $\Ts{3}{v}{4}$.

\begin{ques}
\label{spectrumquestion2}%
What are the implications in   geometrical interpretation of \linebreak
$\ITTs{3}{v}{4}$?
\end{ques}

A Latin bitrade is called {\sf $k$-homogeneous} if each row and
each column contains exactly $k$ elements, and each element
appears exactly $k$ times (see for example~\cite{MR2220235}, for
more information). We may define a $k$-homogeneous
$\T{t}{v}{t+1}$ and seek for their existence.
\begin{ques}
\label{spectrumquestion3}%
What are the possible spectrums of $k$-homogeneous \linebreak
$\ITTs{t}{v}{t+1}$?
\end{ques}
%
\section{Appendix}
%
\begin{center}
A \ $\T{3}{4}{4}$ of volume 17:    \\[.1cm]
\begin{tabular}{|c|c|} \hline
 $T_1$ &
\begin{tabular}
{c@{\hspace{1.9mm}}c@{\hspace{1.9mm}}c@{\hspace{1.9mm}}c@{\hspace{1.9mm}}c@
{\hspace{1.9mm}}c@{\hspace{1.9mm}}c@{\hspace{1.9mm}}c@{\hspace{1.9mm}}c@{\hspace{1.9mm}}c@
{\hspace{1.9mm}}c@{\hspace{1.9mm}}c@{\hspace{1.9mm}}c@{\hspace{1.9mm}}c@{\hspace{1.9mm}}c@
{\hspace{1.9mm}}c@{\hspace{1.9mm}}c@ {\hspace{1.9mm}}c}
 2&2&2&1&1&1&3&3&1&1&3&3&2&2&2&1&1   \\
 3&2&1&3&2&1&3&2&3&2&3&2&3&2&1&2&1   \\
 3&3&3&3&3&3&2&2&2&2&1&1&1&1&1&1&1   \\
 2&3&1&3&1&2&3&2&2&3&2&3&3&1&2&2&1   \\
\end{tabular}
\\ \hline
\end{tabular} \\[.4cm]
\begin{tabular}{|c|c|} \hline
 $T_2$ &
\begin{tabular}
{c@{\hspace{1.9mm}}c@{\hspace{1.9mm}}c@{\hspace{1.9mm}}c@{\hspace{1.9mm}}c@
{\hspace{1.9mm}}c@{\hspace{1.9mm}}c@{\hspace{1.9mm}}c@{\hspace{1.9mm}}c@{\hspace{1.9mm}}c@
{\hspace{1.9mm}}c@{\hspace{1.9mm}}c@{\hspace{1.9mm}}c@{\hspace{1.9mm}}c@{\hspace{1.9mm}}c@
{\hspace{1.9mm}}c@{\hspace{1.9mm}}c@ {\hspace{1.9mm}}c}
 2&2&2&1&1&1&3&3&1&1&3&3&2&2&2&1&1    \\
 3&2&1&3&2&1&3&2&3&2&3&2&3&2&1&2&1    \\
 3&3&3&3&3&3&2&2&2&2&1&1&1&1&1&1&1    \\
 3&1&2&2&3&1&2&3&3&2&3&2&2&3&1&1&2    \\
\end{tabular}
\\ \hline
\end{tabular}
\end{center}
\vspace*{.1cm}
\begin{center}
A \ $\T{3}{4}{4}$ of volume 19:    \\[.1cm]
\begin{tabular}{|c|c|} \hline
 $T_1$ &
\begin{tabular}
{c@{\hspace{1.9mm}}c@{\hspace{1.9mm}}c@{\hspace{1.9mm}}c@{\hspace{1.9mm}}c@
{\hspace{1.9mm}}c@{\hspace{1.9mm}}c@{\hspace{1.9mm}}c@{\hspace{1.9mm}}c@{\hspace{1.9mm}}c@
{\hspace{1.9mm}}c@{\hspace{1.9mm}}c@{\hspace{1.9mm}}c@{\hspace{1.9mm}}c@{\hspace{1.9mm}}c@
{\hspace{1.9mm}}c@{\hspace{1.9mm}}c@{\hspace{1.9mm}}c@
{\hspace{1.9mm}}c@{\hspace{1.9mm}}c}
 3&3&2&2&2&1&1&1&2&2&1&1&3&3&2&2&2&1&1  \\
 3&2&4&3&1&4&2&1&4&2&4&2&3&2&3&2&1&2&1  \\
 3&3&3&3&3&3&3&3&2&2&2&2&1&1&1&1&1&1&1  \\
 3&2&3&2&1&1&3&2&1&3&3&1&2&3&3&1&2&2&1  \\
\end{tabular}
\\ \hline
\end{tabular} \\[.4cm]
\begin{tabular}{|c|c|} \hline
 $T_2$ &
\begin{tabular}
{c@{\hspace{1.9mm}}c@{\hspace{1.9mm}}c@{\hspace{1.9mm}}c@{\hspace{1.9mm}}c@
{\hspace{1.9mm}}c@{\hspace{1.9mm}}c@{\hspace{1.9mm}}c@{\hspace{1.9mm}}c@{\hspace{1.9mm}}c@
{\hspace{1.9mm}}c@{\hspace{1.9mm}}c@{\hspace{1.9mm}}c@{\hspace{1.9mm}}c@{\hspace{1.9mm}}c@
{\hspace{1.9mm}}c@{\hspace{1.9mm}}c@{\hspace{1.9mm}}c@
{\hspace{1.9mm}}c@{\hspace{1.9mm}}c}
 3&3&2&2&2&1&1&1&2&2&1&1&3&3&2&2&2&1&1  \\
 3&2&4&3&1&4&2&1&4&2&4&2&3&2&3&2&1&2&1  \\
 3&3&3&3&3&3&3&3&2&2&2&2&1&1&1&1&1&1&1  \\
 2&3&1&3&2&3&2&1&3&1&1&3&3&2&2&3&1&1&2  \\
\end{tabular}
\\ \hline
\end{tabular}
\end{center}
\vspace*{.1cm}
\begin{center}
A \ $\T{3}{3}{4}$ of volume 21: \\[.1cm]
\begin{tabular}{|c|c|} \hline
 $T_1$ &
\begin{tabular}
{c@{\hspace{1.9mm}}c@{\hspace{1.9mm}}c@{\hspace{1.9mm}}c@{\hspace{1.9mm}}c@
{\hspace{1.9mm}}c@{\hspace{1.9mm}}c@{\hspace{1.9mm}}c@{\hspace{1.9mm}}c@{\hspace{1.9mm}}c@
{\hspace{1.9mm}}c@{\hspace{1.9mm}}c@{\hspace{1.9mm}}c@{\hspace{1.9mm}}c@{\hspace{1.9mm}}c@
{\hspace{1.9mm}}c@{\hspace{1.9mm}}c@{\hspace{1.9mm}}c@{\hspace{1.9mm}}c@{\hspace{1.9mm}}c@
{\hspace{1.9mm}}c@{\hspace{1.9mm}}c}
    3&3&2&2&2&1&1&3&3&2&2&2&1&1&3&3&2&2&2&1&1   \\
    3&2&3&2&1&2&1&3&2&3&2&1&2&1&3&2&3&2&1&2&1   \\
    3&3&3&3&3&3&3&2&2&2&2&2&2&2&1&1&1&1&1&1&1   \\
    1&3&3&2&1&1&2&2&1&1&3&2&2&3&3&2&2&1&3&3&1   \\
\end{tabular}
\\ \hline
\end{tabular} \\[.4cm]
\begin{tabular}{|c|c|} \hline
 $T_2$ &
\begin{tabular}
{c@{\hspace{1.9mm}}c@{\hspace{1.9mm}}c@{\hspace{1.9mm}}c@{\hspace{1.9mm}}c@
{\hspace{1.9mm}}c@{\hspace{1.9mm}}c@{\hspace{1.9mm}}c@{\hspace{1.9mm}}c@{\hspace{1.9mm}}c@
{\hspace{1.9mm}}c@{\hspace{1.9mm}}c@{\hspace{1.9mm}}c@{\hspace{1.9mm}}c@{\hspace{1.9mm}}c@
{\hspace{1.9mm}}c@{\hspace{1.9mm}}c@{\hspace{1.9mm}}c@{\hspace{1.9mm}}c@{\hspace{1.9mm}}c@
{\hspace{1.9mm}}c@{\hspace{1.9mm}}c}
     3&3&2&2&2&1&1&3&3&2&2&2&1&1&3&3&2&2&2&1&1  \\
     3&2&3&2&1&2&1&3&2&3&2&1&2&1&3&2&3&2&1&2&1  \\
     3&3&3&3&3&3&3&2&2&2&2&2&2&2&1&1&1&1&1&1&1  \\
     3&1&1&3&2&2&1&1&2&2&1&3&3&2&2&3&3&2&1&1&3  \\
\end{tabular}
\\ \hline
\end{tabular}
\end{center}
\noindent {\bf Acknowledgement.} This research was in part
supported by a grant from IPM (\#86050213).
%
\def\cprime{$'$}

\end{document}